%% file: main.tex
\documentclass{article}

\usepackage[T1]{fontenc}
\usepackage[utf8]{inputenc}

\usepackage{subcaption}

\usepackage{adjustbox}

\usepackage{dsfont} 
\usepackage{amssymb}

\usepackage{amsmath}
\usepackage{amsthm}
\newtheorem{theorem}{Theorem}
\newtheorem{lemma}[theorem]{Lemma}

\newtheorem{proposition}[theorem]{Proposition}
\newtheorem{conjecture}[theorem]{Conjecture}

\usepackage[nocompress]{cite}

\usepackage{tikz}
\usetikzlibrary{
    calc,
    positioning,
    backgrounds,
}
\usepackage{pgfplots}
\pgfplotsset{compat=1.14}

\pgfplotsset{
    closed dot/.style = {
        only marks,
        mark = *,
        mark size = 1.5pt,
    },
    open dot/.style = {
        only marks,
        mark = *,
        fill = white,
        mark size = 1.5pt,
    },
}

\newcommand{\norm}[1]{\left\lVert #1 \right\rVert}
\DeclareMathOperator{\vol}{vol} 
\DeclareMathOperator{\area}{area} 
\DeclareMathOperator{\ppyr}{ppyr} 

\usepackage[hidelinks]{hyperref}

\makeatletter
\newcommand{\restate@theorem@name}{}
\newtheorem*{restate@theorem}{\restate@theorem@name}
\newenvironment{restatetheorem}[1]{
    \renewcommand{\restate@theorem@name}{Theorem~\ref{#1}}
    \begin{restate@theorem}
}{
    \end{restate@theorem}
}

\begin{document}

\title{
    Reconstruction of symmetric convex bodies \\
    from Ehrhart-like data
}
\author{
    Tiago Royer
}
\date{}
\maketitle

\begin{abstract}
    In a previous paper~\cite{Eu},
    we showed how to use the Ehrhart function $L_P(s)$,
    defined by $L_P(s) = \#(sP \cap \mathbb Z^d)$,
    to reconstruct a polytope $P$.
    More specifically,
    we showed that,
    for rational polytopes $P$ and $Q$,
    if $L_{P + w}(s) = L_{Q + w}(s)$ for all integer vectors $w$,
    then $P = Q$.
    In this paper we show the same result,
    but assuming that $P$ and $Q$ are symmetric convex bodies
    instead of rational polytopes.
\end{abstract}

\input{introduction}

\input{notation}

\input{projection-areas}
\input{approximating-spherical-projections}
\input{limit-behavior-pseudopyramids}
\input{piecing-everything}

\input{final-remarks}

\bibliographystyle{plainurl}
\bibliography{bib}

\end{document}

%% file: introduction.tex
\section{Introduction}

Given a polytope $P \subseteq \mathbb R^d$,
the classical Ehrhart lattice point enumerator $L_P(t)$ is defined as
\begin{equation*}
    L_P(t) = \#(tP \cap \mathbb Z^d), \qquad \text{integer $t \geq 0$.}
\end{equation*}
Here,
$\#(A)$ is the number of elements in $A$
and $tP = \{tx \mid x \in P\}$ is the dilation of $P$ by $t$.
The above definition may be extended
to allow for $P$ to be an arbitrary convex body,
and for $t$ to be an arbitrary real number.

To minimize confusion,
we will denote real dilation parameters with the letter $s$,
so that $L_P(t)$ denotes the classical Ehrhart function
and $L_P(s)$ denotes the extension considered in this paper.
So,
for example,
$L_P(t)$ is just the restriction of $L_P(s)$ to integer values.

It is clear from the definition that
the classical Ehrhart function is invariant under integer translations;
that is,
for every polytope $P$ and every integer vector $w$,
we have
\begin{equation*}
    L_{P + w}(t) = L_P(t)
\end{equation*}
for all $t$.
This is not true for the real Ehrhart function $L_P(s)$.
In fact,
an earlier paper~\cite{Eu}
shows that the list of functions $L_{P + w}(s)$,
for integer $w$,
is a complete set of invariants for full-dimensional rational polytopes.
More precisely:

\begin{theorem}
    \label{thm:rational-reconstruction}
    Let $P$ and $Q$ be full-dimensional rational polytopes,
    and suppose that $L_{P + w}(s) = L_{Q + w}(s)$
    for all $s > 0$ and all integer $w$.
    Then $P = Q$.
\end{theorem}

It is also conjectured that
the rationality hypothesys may be dropped.
In this paper,
we will show a special case of this conjecture,
where we assume that the polytopes are symmetric
(that is, $x \in P$ if and only if $-x \in P$).
In fact,
we never use the fact that $P$ and $Q$ are polytopes;
we just need convexity.
So we have the following.

\begin{theorem}
    \label{thm:convex-body-translation-variant}
    Let $K$ and $H$ be symmetric convex bodies,
    and assume that $L_{K + w}(s) = L_{H + w}(s)$
    for all real $s > 0$ and all integer $w$.
    Then $K = H$.
\end{theorem}

%% file: notation.tex
\section{Notation and basic results}

The punchline is Alexsandrov's projection theorem.
Let $K \subseteq \mathbb R^d$ be a convex body
(that is,
a covex, compact set with nonempty interior).
For any unit vector $v$,
we will denote by $V_K(v)$
the $(d-1)$-dimensional area of the orthogonal projection of $K$ in $\{v\}^\perp$.

For example,
let $K = [0, 1] \times [0, 1] \subseteq \mathbb R^2$,
$v = (0, 1)$ and $v' = (\frac{\sqrt 2}{2}, \frac{\sqrt 2}{2})$.
Then $V_K(v) = 1$ and $V_K(v') = \sqrt 2$.

A convex body $K$ is said to be \emph{symmetric}
if $x \in K$ if and only if $-x \in K$.
An importan reconstruction theorem for symmetric convex bodies
is Aleksandrov's projecion theorem
(see e.g.~\cite[p.~115]{GeometricTomography}).

\begin{theorem}[Aleksandrov's projection theorem]
    Let $K$ and $H$ be two symmetric convex bodies in $\mathbb R^d$
    such that $V_K(v) = V_H(v)$ for all unit vectors $v$.
    Then $K = H$.
\end{theorem}

So,
the goal is to compute the function $V_K$ using the Ehrhart functions $L_{K + w}$.
The two main tools are the Hausdorff distance and pseudopyramids.

For $\lambda \geq 0$ and $x \in \mathbb R^d$,
denote by $B_\lambda(x)$ the ball with radius $\lambda$ centered at $x$.
If $K$ is a convex body and $\lambda \geq 0$,
define $K_\lambda$ by
\begin{equation*}
    K_\lambda = \bigcup_{x \in K} B_\lambda(x).
\end{equation*}

The Hausdorff distance $\rho(K, H)$ between two convex bodies $K$ and $H$
is defined to be (Figure~\ref{fig:hausdorff-distance})
\begin{equation*}
    \rho(K, H) = \inf \{ \lambda \geq 0 \mid
        K \subseteq H_\lambda \text{ and } H \subseteq K_\lambda
    \}.
\end{equation*}

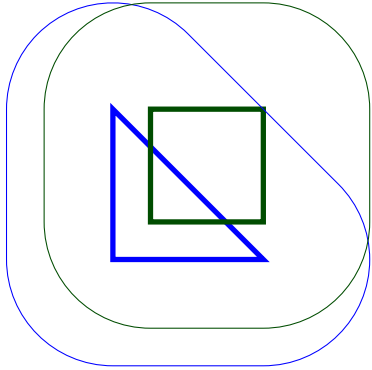
\begin{figure}[t]
    \centering
    \begin{tikzpicture}
        \draw[line width = 2pt, blue] (0, 0) -- (0, 2) -- (2, 0) -- cycle;
        \draw[line width = 2pt, green!30!black]
            (0.5, 0.5) -- (2, 0.5) -- (2, 2) -- (0.5, 2) -- cycle;

        \pgfmathsetmacro{\rho}{sqrt(2)}
        \draw[thin, blue]
            (-\rho, 0) arc [radius = \rho, start angle = 180, end angle = 270]
            -- ++(2, 0) arc [radius = \rho, start angle = -90, end angle = 45]
            -- ++(-2, 2) arc [radius = \rho, start angle = 45, end angle = 180]
            -- cycle;

        \draw[thin, green!30!black] (0.5, 0.5)
            ++(-\rho, 0) arc [radius = \rho, start angle = 180, end angle = 270]
            -- ++(1.5, 0) arc [radius = \rho, start angle = -90, end angle = 0]
            -- ++(0, 1.5) arc [radius = \rho, start angle = 0, end angle = 90]
            -- ++(-1.5, 0) arc [radius = \rho, start angle = 90, end angle = 180]
            -- cycle;
    \end{tikzpicture}
    \caption[
        Hausdorff distance between two convex sets.
    ]{
        Hausdorff distance between two convex sets.
        The thick lines are the boundaries of the sets $K$ and $H$;
        the thin lines are the boundaries of the sets
        $K_\lambda$ and $H_\lambda$.
    }
    \label{fig:hausdorff-distance}
\end{figure}

It can be shown that
the set of convex sets in $\mathbb R^d$
is a metric space under the Hausdorff distance
and that the Euclidean volume is continuous in this space
(see e.g.~\cite[p.~9]{GeometricTomography}),
but we just need the following special case of this theory.

\begin{lemma}
    \label{thm:hausdorff-distance}
    Let $K$ and $A_1, A_2, \dots$ be convex bodies.
    If $\lim_{i \to \infty} \rho(K, A_i) = 0$,
    then $\lim_{i \to \infty} \vol A_i = \vol K$.
\end{lemma}

The concept of pseudopyramid was introduced in~\cite{Eu}.
If $K$ is a convex body,
the pseudopyramid $\ppyr K$ is defined to be
(Figure~\ref{fig:pseudopyramid})
\begin{equation*}
    \ppyr K = \bigcup_{0 \leq \lambda \leq 1} \lambda K.
\end{equation*}

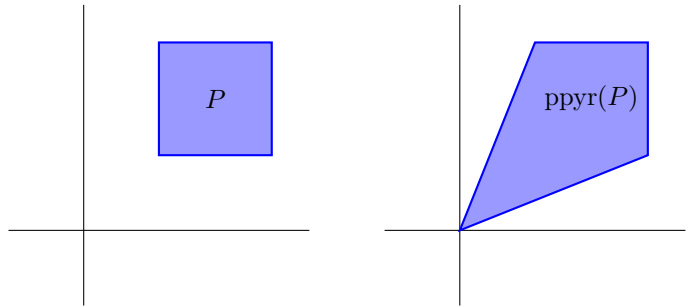
\begin{figure}[t]
    \centering
    \begin{tikzpicture}
        \draw (-1, 0) -- (3, 0);
        \draw (0, -1) -- (0, 3); 
        \filldraw [thick, blue, fill = blue!40] (1, 1) rectangle (2.5, 2.5);
        \node at (1.75, 1.75) {$P$};

        \begin{scope}[xshift = 5cm]
            \draw (-1, 0) -- (3, 0);
            \draw (0, -1) -- (0, 3); 
            \filldraw [thick, blue, fill = blue!40]
                (0, 0) -- (1, 2.5) -- (2.5, 2.5) -- (2.5, 1) -- cycle;
            \node at (1.75, 1.75) {$\ppyr(P)$};
        \end{scope}
    \end{tikzpicture}
    \caption{
        Pseudopyramid of a polytope.
    }
    \label{fig:pseudopyramid}
\end{figure}

We will use the following lemma,
which is a consequence of Lemma 1 of~\cite{Eu}.

\begin{lemma}
    \label{thm:volume-of-pseudopyramid}
    Let $K$ and $H$ be convex bodies,
    and suppose that $L_K(s) = L_H(s)$ for all $s > 0$.
    Then $\vol \ppyr H = \vol \ppyr K$.
\end{lemma}

In other words,
whenever we know $L_K$,
we may assume we also know $\vol \ppyr K$.

\begin{proof}
    Lemma 1 of~\cite{Eu} states that,\footnote{
        Tecnically,
        Lemma 1 of~\cite{Eu} only states this for polytopes,
        but the proof holds verbatim for convex bodies.
    }
    if $L_K(s) = L_H(s)$,
    then $L_{\ppyr K}(s) = L_{\ppyr H}(s)$.
    Since
    \begin{equation*}
        \lim_{s \to \infty} \frac{L_{\ppyr K}(s)}{s^d} = \vol \ppyr K
    \end{equation*}
    (and similarly for $H$),
    we have $\vol \ppyr K = \vol \ppyr H$.
\end{proof}

%% file: projection-areas.tex
\section{Pseudopyramid volumes and areas of projections}
\label{sec:projection-areas}

In this section,
we will show how to compute the function $V_K$
in terms of the numbers $\vol \ppyr(P + w)$ for integer $w$.

Given a convex body $K$,
define its spherical projection $S(K)$ by
(Figure~\ref{fig:spherical-projection})
\begin{equation*}
    S(K) = \left\{ \frac{x}{\norm x}
        \,\middle|\, x \in K \text{ and } x \neq 0 \right
    \}.
\end{equation*}

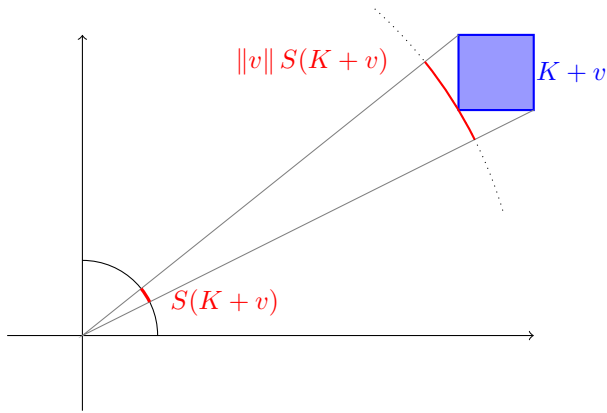
\begin{figure}[t]
    \centering
    \begin{tikzpicture}
        \draw[->] (-1, 0) -- (6, 0);
        \draw[->] (0, -1) -- (0, 4);

        \draw (1, 0) arc [start angle = 0, end angle = 90, radius = 1];

        \filldraw [thick, blue, fill = blue!40]
            (5, 3) -- (6, 3) -- (6, 4) -- (5, 4) -- cycle;
        \node[blue] at (6.5, 3.5) {$K + v$};

        \draw [gray, thin] (5, 4) -- (0, 0) -- (6, 3);

        \pgfmathsetmacro{\startangle}{atan(3/6)}
        \pgfmathsetmacro{\endangle}{atan(4/5)}
        \pgfmathsetmacro{\radius}{veclen(5, 3)}

        \draw [red, very thick] (\startangle:1)
            arc [start angle = \startangle, end angle = \endangle, radius = 1];
        \draw [red] (\startangle:1)
            +(1, 0) node {$S(K + v)$};

        \draw [dotted] (\startangle - 10:\radius)
            arc [start angle = \startangle - 10,
                end angle = \endangle + 10, radius = \radius];
        \draw [red, thick] (\startangle : \radius)
            arc [start angle = \startangle, end angle = \endangle, radius = \radius]
            ++(-1.5, 0) node {$\norm v S(K + v)$};
    \end{tikzpicture}
    \caption{
        Spherical projection of $K + v$.
    }
    \label{fig:spherical-projection}
\end{figure}

The connection between pseudopyramid volumes
and areas of projections
can be seen in Figure~\ref{fig:spherical-projection}.
The set $\norm v S(K + v)$ is a dilation of the projection $S(K + v)$ of $K$.
Note that the shape of $\norm v S(K + v)$
``looks like'' the orthogonal projection of $K$ in $\{v\}^\perp$;
that is,
the area of $\norm v S(K + v)$
approximates $V_K(v)$.

If the pseudopyramid were an actual pyramid
(with base $\norm v S(K + v)$),
then using the formula $v = \frac{Ah}{d}$ for the volume of a pyramid
would allow us to discover what is the area of the projection,
which would give an approximation to $V_K(v)$.
We will show that this formula is true ``in the infinity'';
that is,
using limits,
we can recover the area of the projection
using taller and taller pseudopyramids.

%% file: approximating-spherical-projections.tex
\subsection{Approximating spherical projections}

For convex bodies $K$,
the set $S(K)$ is a manifold\footnote{
    Technically,
    $S(K)$ will be a manifold-with-corners
    (see~\cite[p.~137]{SpivakManifolds}).
    However,
    their interiors relative to the $(d-1)$-dimensional sphere $S^{d-1}$
    are manifolds,
    and since we're dealing with areas
    there will be no harm in ignoring these boundaries.
}.
If $K$ does not contain the origin in its interior,
then $S(K)$ may be parameterized with a single coordinate system;
that is,
there is a set $U \subseteq \mathbb R^{d-1}$
and a continuously differentiable function $\varphi: U \to S(K)$
which is a bijection between $U$ and $S(K)$.
Since we want to move $P$ towards infinity,
this shall always be the case
if the translation vector is long enough.
In this case,
we define its area to be~\cite[p.~126]{SpivakManifolds}
\begin{align*}
    \area S(K) &= \int_U \norm{D_1 \varphi \times \dots \times D_{d-1} \varphi} \\
        &= \int_U \norm{
            \frac{\partial \varphi}{\partial x_1}
            \times \dots \times
            \frac{\partial \varphi}{\partial x_{d-1}} dx_1 \dots dx_{d-1}
        }.
\end{align*}

The following theroem states that
the spherical projection approximates,
in a sense,
the orthogonal projection,
for large enough translation vectors.

\begin{theorem}
    \label{thm:spherical-approaches-orthogonal}
    Let $v$ be a unit vector
    and $K \subseteq \mathbb R^d$ a convex body.
    Then
    \begin{equation*}
        \lim_{\mu \to \infty} \mu^{d-1} \area S(P + \mu v) = V_K(v).
    \end{equation*}
\end{theorem}

\begin{proof}
    By rotating all objects involved if needed,
    we may assume that $v = (0, \dots, 0, 1)$.
    Let $N$ be large enough that $K \subseteq B_N(0)$;
    we'll assume that $\mu > N$,
    so that $K + \mu v$ lies strictly above the hyperplane $x_d = 0$.

    Let $K'$ be the orthogonal projection of $K$ into $\{v\}^\perp$.
    We'll think of $K'$ as being a subset of $\mathbb R^{d-1}$.
    Denote by $K_\mu$ the projection of the set $\mu S(K + \mu v)$ on $\mathbb R^{d-1}$
    (Figure~\ref{fig:spherical-orthogonal-projection});
    that is,
    first project $K + \mu v$ to the sphere with radius $\mu$,
    then discard the last coordinate.
    Nnote this is similar to project it to the hyperplane $x_d = 0$.
    We'll show that,
    as $\mu$ goes to infinity,
    both the Hausdorff distance between $K_\mu$ and $K'$
    and the difference between the volume of $K_\mu$ and the area of $\mu S(K + \mu v)$
    tend to zero.

    \begin{figure}[t]
        \centering
        \begin{tikzpicture}
            \pgfmathsetmacro{\Mu}{4}
            \pgfmathsetmacro{\N}{2}
            \draw[gray, thin] (-\Mu - 1, 0) -- (\Mu + 1, 0);

            \draw[dashed] (\Mu, 0) arc [radius = \Mu, start angle = 0, end angle = 180];
            \draw[|<->|] (-\Mu, 0) -- node [below] {$\mu$} (0, 0);

            \draw[dashed] (0, \Mu) circle [radius = \N];
            \draw[|<->|] (-\N, \Mu) -- node [above] {$N$} (0, \Mu);

            \coordinate (A) at (0.6, \Mu + 1.0);
            \coordinate (B) at (1.6, \Mu + 0.8);
            \coordinate (C) at (1.8, \Mu + 0.4);
            \coordinate (D) at (0.8, \Mu + 0.4);
            \draw [thick, blue, fill = blue!40] (A) -- (B) -- (C) -- (D) -- cycle;
            \node [blue] at (\N + 0.5, \Mu + 0.5) {$K + \mu v$};

            \draw [blue, line width = 2pt, yshift = 1pt]
                (A |- 0, 0) -- node [above] {$K'$} (C |- 0, 0);

            \pgfmathanglebetweenpoints{\pgfpointorigin}{\pgfpointanchor{C}{center}}
            \edef\startangle{\pgfmathresult}
            \pgfmathanglebetweenpoints{\pgfpointorigin}{\pgfpointanchor{A}{center}}
            \edef\endangle{\pgfmathresult}
            \draw [very thick, red] (\startangle:\Mu)
                coordinate (S-right)
                arc [radius = \Mu, start angle = \startangle, end angle = \endangle]
                coordinate (S-left);
            \node [red, right = 0.5] at (S-right) {$\mu S(K + \mu v)$};

            \draw [green!30!black, line width = 2pt, yshift = -1pt]
                (S-left |- 0, 0) -- node [below] {$K_\mu$} (S-right |- 0, 0);

            \begin{scope}[very thin]
                \draw [red] (0, 0) -- (A);
                \draw [red] (0, 0) -- (C);

                \draw [blue] (A) -- (A |- 0, 0);
                \draw [blue] (C) -- (C |- 0, 0);
                \draw [green!30!black] (S-left) -- (S-left |- 0, 0);
                \draw [green!30!black] (S-right) -- (S-right |- 0, 0);
            \end{scope}
        \end{tikzpicture}
        \caption[
            Spherical and orthogonal projections.
        ]{
            The spherical projection $\mu S(K + \mu v)$,
            when projected orthogonally to the plane $x_d = 0$
            (the set $K_\mu$),
            approaches the volume of the projection $K'$.
        }
        \label{fig:spherical-orthogonal-projection}
    \end{figure}
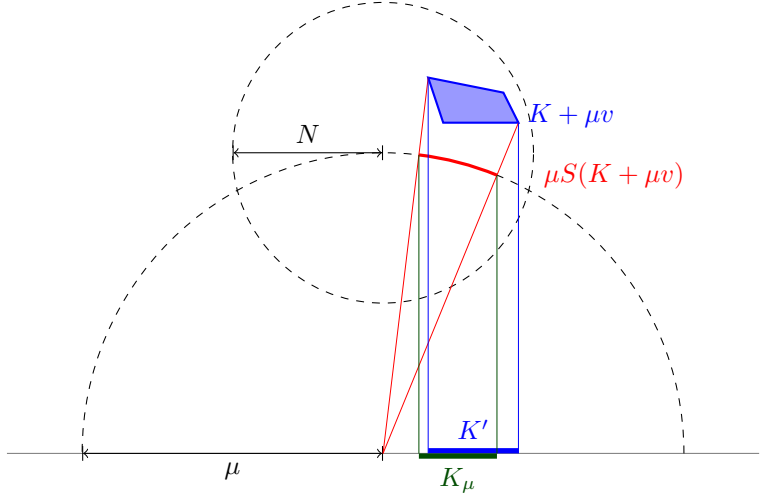

    First, let's bound the Hausdorff distance between $K'$ and $K_\mu$.
    If $x \in K + \mu v$,
    then $x$ gets projected to a point $x_0 \in K'$
    by just discarding the last coordinate;
    however, to be projected to a point $x_1 \in K_\mu$,
    first we replace $x$ by $x' = \frac{\mu}{\norm x} x$
    to get a point $x' \in \mu S(K + \mu v)$,
    and then the last coordinate of $x'$ is discarded.
    Note that $x_1 = \frac{\mu}{\norm x} x_0$;
    therefore,
    the distance between these two points is
    \begin{equation*}
        \norm{x_0 - x_1}
            = \left| 1 - \frac{\mu}{\norm x} \right| \norm{x_0}
            = \frac{|\norm x - \mu|}{\norm x} \norm{x_0}
    \end{equation*}

    We have $x \in K + \mu v \subseteq B_N(\mu v)$,
    so $\mu - N \leq \norm x \leq \mu + N$.
    As $v = (0, \dots, 0, 1)$ and $x_0$ is $x$ without its last coordinate,
    we have $\norm{x_0} \leq N$
    (because, in $\mathbb R^{d-1}$,
    we have $x_0 \in B_N(0)$).
    So,
    the distance between $x_0$ and $x_1$ is at most $\frac{N^2}{\mu + N}$.

    Every point in $K'$ and in $K_\mu$ is obtained through these projections.
    This means that,
    given any point $x_0$ in one of the sets,
    we may find another point $x_1$ in the other set
    which is at a distance of at most $\frac{N^2}{\mu + N}$ from the former,
    because we can just pick a point $x$ whose projection is $x_0$;
    then its other projection $x_1$ will be close to $x_0$.
    Thus
    \begin{equation*}
        \rho(K', K_\mu) \leq \frac{N^2}{\mu + N},
    \end{equation*}
    so by Theorem~\ref{thm:hausdorff-distance}
    the volumes of $K_\mu$ converges to $\vol K'$.

    Now,
    let's relate $\vol K_\mu$ with $\mu^{d-1} \area S(K + \mu v)$.
    If $y = (y_1, \dots, y_d)$ is a point in $\mu S(K + \mu v)$,
    we know that $\norm y = \mu$
    and that $y_d > 0$
    (because we're assuming $\mu > N$).
    Therefore,
    if we define $\varphi: K_\mu \to \mu S(K + \mu v)$ by
    \begin{equation*}
        \varphi(y_1, \dots, y_{d-1}) = \left(
            y_1, \dots, y_{d-1}, \sqrt{\mu^2 - y_1^2 - \dots - y_{d-1}^2}
        \right),
    \end{equation*}
    then $\varphi$ will be a differentiable bijection
    between $K_\mu$ and $\mu S(K + \mu v)$,
    so that $\varphi$ is a parametrization for $\mu S(K + \mu v)$.

    For the partial derivatives,
    we have $\frac{\partial \varphi_i}{\partial y_j} = [i = j]$ if $i < d$;
    that is,
    the partial derivatives behave like the identity.
    For $i = d$,
    we have
    \begin{equation*}
        \frac{\partial \varphi_d}{\partial y_j} =
            \frac{y_j}{\sqrt{\mu^2 - y_1^2 - \dots - y_{d-1}^2}}.
    \end{equation*}

    Now, by definition of $N$, we have
    \begin{equation*}
        \left| \frac{\partial \varphi_d}{\partial y_j} \right|
            \leq \frac{N}{\sqrt{\mu^2 - N^2}},
    \end{equation*}
    so the vectors $D_i \varphi$ converge uniformly to $e_i$ for large $\mu$.
    Since the generalized cross product is linear in each entry,
    the vector $D_1 \varphi \times \dots \times D_{d-1} \varphi$
    converges uniformly to $e_d$,
    and thus the number
    \begin{align*}
        \left| \vol K_\mu - \area \mu S(K + \mu v) \right|
        &=
            \left| \int_{K_\mu} 1 -
                \int_{K_\mu} \norm{D_1 \varphi \times \dots \times D_{d-1} \varphi}
            \right| \\
        &\leq
            \int_{K_\mu} \Big|
                1 - \norm{D_1 \varphi \times \dots \times D_{d-1} \varphi}
            \Big|
    \end{align*}
    converges to zero.

    Now, since $\area (\mu S(K + \mu v)) = \mu^{d-1} \area S(K + \mu v)$,
    combining these two convergence results gives the theorem.
\end{proof}

%% file: limit-behavior-pseudopyramids.tex
\subsection{Limit behavior of pseudopyramids}

Now we'll show how to use the pseudopyramids to compute these projections.

Let $K$ be a pseudopyramid.
Define the \emph{outer radius} $R(K)$ of $K$
to be the smallest number such that
the ball of radius $R(K)$ around the origin contains $K$.
That is,
\begin{equation*}
    R(K) = \inf \{R \geq 0 \mid K \subseteq B_R(0) \}.
\end{equation*}

Define the \emph{front shell} of $K$
to be the set of points in the boundary of $K$
which are not contained in any facet passing through the origin;
that is,
the set of points $x$ in the boundary of $K$ such that
$\lambda x$ is contained in the interior of $K$ for all $0 < \lambda < 1$.
Define, then,
the \emph{inner radius} $r(K)$ of $K$
to be the largest number such that
the ball of radius $r(K)$ around the origin
contains no points of the front shell of $K$
(Figure~\ref{fig:inner-outer-radius}).
Note that this is equivalent to $r(K) S(K)$ to be contained in $K$;
that is,
\begin{equation*}
    r(K) = \sup \{r \geq 0 \mid r S(K) \subseteq K\}.
\end{equation*}

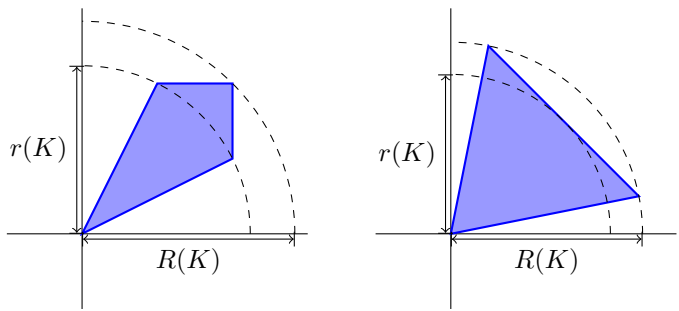
\begin{figure}[t]
    \centering
    \begin{tikzpicture}
        \draw (-1, 0) -- (3, 0);
        \draw (0, -1) -- (0, 3);
        \draw [thick, blue, fill = blue!40]
            (0, 0) -- (2, 1) -- (2, 2) -- (1, 2) -- cycle;

        \pgfmathsetmacro{\innerradius}{veclen(2, 1)}
        \pgfmathsetmacro{\outerradius}{veclen(2, 2)}
        \draw [dashed] (\innerradius, 0) arc
            [radius = \innerradius, start angle = 0, end angle = 90];
        \draw [dashed] (\outerradius, 0) arc
            [radius = \outerradius, start angle = 0, end angle = 90];

        \draw [xshift = -2pt, |<->|]
            (0, 0) -- node[left] {$r(K)$} (0, \innerradius);
        \draw [yshift = -2pt, |<->|]
            (0, 0) -- node[below] {$R(K)$} (\outerradius, 0);
    \end{tikzpicture}
    \qquad
    \begin{tikzpicture}
        \draw (-1, 0) -- (3, 0);
        \draw (0, -1) -- (0, 3);
        \draw [thick, blue, fill = blue!40]
            (0, 0) -- (2.5, 0.5) -- (0.5, 2.5) -- cycle;

        \pgfmathsetmacro{\innerradius}{veclen(1.5, 1.5)}
        \pgfmathsetmacro{\outerradius}{veclen(2.5, 0.5)}
        \draw [dashed] (\innerradius, 0) arc
            [radius = \innerradius, start angle = 0, end angle = 90];
        \draw [dashed] (\outerradius, 0) arc
            [radius = \outerradius, start angle = 0, end angle = 90];

        \draw [xshift = -2pt, |<->|]
            (0, 0) -- node[left] {$r(K)$} (0, \innerradius);
        \draw [yshift = -2pt, |<->|]
            (0, 0) -- node[below] {$R(K)$} (\outerradius, 0);
    \end{tikzpicture}
    \caption{
        Inner and outer radius for two pseudopyramids.
    }
    \label{fig:inner-outer-radius}
\end{figure}

We leave the following lemma to the reader.
It relates the inner and outer radii
with the area of the spherical projection.

\begin{proposition}
    \label{thm:radii-vs-spherical-projection}
    Let $K \subseteq \mathbb R^d$ be a convex body
    which does not contain the origin.
    Then
    \begin{equation*}
        \frac{\vol \ppyr K}{R(\ppyr K)^d}
        \leq
        \frac{\area S(K)}{d}
        \leq
        \frac{\vol \ppyr K}{r(\ppyr K)^d}.
    \end{equation*}
\end{proposition}

This lemma,
combined with Lemma~\ref{thm:spherical-approaches-orthogonal},
shows how to calculate the volume of the orthogonal projection
knowing only the volumes of the pseudopyramids.

\begin{lemma}
    \label{thm:volume-projection-from-volume-ppyr}
    Let $K \subseteq \mathbb R^d$ be a convex body,
    and $v$ any unit vector.
    Then
    \begin{equation*}
        \lim_{\mu \to \infty}
            \frac{\vol \ppyr(K + \mu v)}{\mu}
            = \frac{V_K(v)}{d}.
    \end{equation*}
\end{lemma}

\begin{proof}
    Let $N$ be large enough so that $K \subseteq B_N(0)$.
    For any $\mu$,
    as $v$ is a unit vector,
    we have $K + \mu v \subseteq B_{N + \mu}(0)$,
    so
    \begin{equation*}
        R(\ppyr(K + \mu v)) \leq \mu + N.
    \end{equation*}

    Since all the points in the front shell of $\ppyr(K + \mu v)$
    are points of $K$,
    all of them must have norm greater or equal to $\mu - N$.
    Therefore,
    no origin-centered ball with radius smaller than that can contain these points.
    Thus,
    \begin{equation*}
        r(\ppyr(K + \mu v)) \geq \mu - N.
    \end{equation*}

    Using these two inequalities
    and Proposition~\ref{thm:radii-vs-spherical-projection}
    gives
    \begin{equation*}
        \frac{\vol \ppyr(K + \mu v)}{(\mu + N)^d}
        \leq
        \frac{\area S(K + \mu v)}{d}
        \leq
        \frac{\vol \ppyr(K + \mu v)}{(\mu - N)^d},
    \end{equation*}
    which may be rewritten as
    \begin{equation*}
        \begin{adjustbox}{center} 
            $\displaystyle
                \frac{(\mu - N)^d}{\mu^d} \frac{\mu^{d-1} \area S(K + \mu v)}{d}
                \leq
                \frac{\vol \ppyr(K + \mu v)}{\mu}
                \leq
                \frac{(\mu + N)^d}{\mu^d} \frac{\mu^{d-1} \area S(K + \mu v)}{d}.
            $
        \end{adjustbox}
    \end{equation*}

    Now Theorem~\ref{thm:spherical-approaches-orthogonal}
    and the squeeze theorem finish the proof.
\end{proof}

For example,
for $K = [0, 1]^2$ and $v = (1, 0)$,
we have $\vol \ppyr(K + \mu v) = 1 + \frac{\mu}{2}$,
so $\lim_{\mu \to \infty} \frac{\vol \ppyr(K + \mu v)}{\mu} = \frac12$,
which is precisely one-half of the area of $\{0\} \times [0, 1]$,
the projection $K'$ of $K$ on the $y$-axis.
This highlights that,
for large $\mu$,
the pseudopyramid $\ppyr(K + \mu v)$ ``behaves like'' an actual pyramid,
with height $\mu$ and base $K'$.

%% file: piecing-everything.tex
\subsection{Piecing everything together}

\begin{restatetheorem}{thm:convex-body-translation-variant}
    Let $K$ and $H$ be symmetric convex bodies,
    and assume that $L_{K + w}(s) = L_{H + w}(s)$
    for all real $s > 0$ and all integer $w$.
    Then $K = H$.
\end{restatetheorem}

\begin{proof}
    By Lemma~\ref{thm:volume-of-pseudopyramid},
    we have $\vol \ppyr (K + w) = \vol \ppyr (H + w)$
    for all integer $w$.

    If $w$ is a nonzero integer vector,
    let $v = \frac{w}{\norm w}$;
    then,
    by Lemma~\ref{thm:volume-projection-from-volume-ppyr},
    we have
    \begin{align*}
        \frac{V_K(v)}{d}
            &= \lim_{\mu \to \infty} \frac{\vol \ppyr(K + \mu v)}{\mu} \\
            &= \lim_{\mu \to \infty} \frac{\vol \ppyr(H + \mu v)}{\mu} \\
            &= \frac{V_H(v)}{d}
    \end{align*}

    This shows that,
    whenever $v$ is a multiple of a rational vector,
    we have $V_K(v) = V_H(v)$.
    Since the functiosn $V_K$ and $V_H$ are continuous,
    we have $V_K = V_H$,
    and thus by Aleksandrov's projection theorem
    we conclude that $K = H$.
\end{proof}

%% file: final-remarks.tex
\section{Final remarks}

Theorems \ref{thm:rational-reconstruction} and~\ref{thm:convex-body-translation-variant}
both assume that $L_{K + w}(s) = L_{H + w}(s)$
for all integer $w$ and real $s > 0$,
and both conclude that $K = H$.
The first theorem assume that the objects being considered are rational polytopes,
but no symmetry condition is imposed;
the second theorem assumes that the objects are symmetric,
but otherwise permits arbitrary convex bodies.

This suggest the following common generalization of these theorems:

\begin{conjecture}
    Let $K$ and $H$ be any convex bodies,
    and assume that $L_{K + w}(s) = L_{H + w}(s)$
    for all real $s > 0$ and all integer $w$.
    Then $K = H$.
\end{conjecture}